\documentclass[11pt,english]{article}
\usepackage[letterpaper]{geometry}
\usepackage{amsmath,amssymb,amsthm}
\usepackage{graphicx}
\usepackage{epstopdf}
\usepackage{setspace}
\usepackage{babel}

\theoremstyle{plain}
\newtheorem{theorem}{Theorem}
\newtheorem*{theorem*}{Theorem}

\newtheorem{corollary}{Corollary}
\newtheorem*{corollary*}{Corollary}
\newtheorem{definition}{Definition}
\newtheorem*{remark*}{Remark}

\def\mathscr{\mathfrak}

\def\N{{\mathbb N}}
\def\Z{{\mathbb Z}}

\def\-{\backslash}

\makeatletter

\date{}

\makeatother

\begin{document}

\title{Non-classical linear divisibility sequences\\ and cyclotomic polynomials}

\author{Sergiy Koshkin\\
 Department of Mathematics and Statistics\\
 University of Houston-Downtown\\
 One Main Street\\
 Houston, TX 77002\\
 e-mail: koshkins@uhd.edu}
\maketitle
\begin{abstract}
Divisibility sequences are defined by the property that their elements divide each other whenever their indices do. The divisibility sequences that also satisfy a linear recurrence, like the Fibonacci numbers, are generated by polynomials that divide their compositions with every positive integer power. We completely characterize such polynomials in terms of their factorizations into cyclotomic polynomials using labeled Hasse diagrams, and construct new integer divisibility sequences based on them. We also show that, unlike the Fibonacci numbers, these non-classical sequences do not have the property of strong divisibility.
 
\bigskip

\end{abstract}

\section{What Mersenne and Fibonacci have in common}\label{S1}

The Mersenne numbers are given explicitly as $M_n:=2^n-1$. The Fibonacci numbers, on the other hand, come from a recurrence relation, $F_{n+2}=F_{n+1}+F_n$, with $F_0=0$ and $F_1=1$. Of course, $M_0=0$ and $M_1=1$, and $M_n$ satisfy a similar recurrence, $M_{n+2}=3M_{n+1}-2M_n$. But both sequences also have a more interesting property in common: if $m$ divides $n$ then $a_m$ divides $a_n$. For example, $F_7=13$ and $F_{14}=377$, and we observe that $377=13\cdot29$. Such sequences are called {\it divisibility sequences} (the name seems to be due to Hall \cite{Hall}), and {\it linear divisibility sequences} if they also satisfy a linear recurrence $a_{n+k}=c_1a_{n+k-1}+\cdots+c_ka_n$, like the Mersenne or Fibonacci numbers. 

Lucas was first to study in depth the second order ($k=2$) linear divisibility sequences in 1876-1880. He showed that, up to normalization, non-trivial such sequences are of the form $\frac{\alpha^n-\beta^n}{\alpha-\beta}$, where $\alpha$ and $\beta$ are the roots of a quadratic polynomial with integer coefficients. The polynomial is then $x^2-(\alpha+\beta)x+\alpha\beta$, and the recurrence is $a_{n+2}=(\alpha+\beta)a_{n+1}-\alpha\beta a_n$, so the sequence consists of integers even if $\alpha$ and $\beta$ are irrational. For the Fibonacci numbers, the Lucas's form is given by the Binet formula, $F_n=\frac{(\frac{1+\sqrt{5}}2)^n-(\frac{1-\sqrt{5}}2)^{n}}{\sqrt{5}}$. Based on his work with divisibility sequences, Lucas devised primality tests that helped establish the primality of $M_{127}$. The search for large Mersenne primes is ongoing, now with the help of the Internet.

But what is behind the peculiar divisibility property? This question leads from numerical sequences to sequences of polynomials. It still makes sense to talk about divisibility sequences of polynomials, e.g. $x^n$ is one. Slightly less obviously, so is $x^n-1$, because $x^{dm}-1$ has $x^m-1$ as a factor by the difference of powers formula. The Fibonacci polynomials are a more non-trivial example \cite{WP}. Up to normalization by $\frac1{\alpha-\beta}$, the Lucas sequences are products of $\beta^n$ and $(\frac{\alpha}{\beta})^n-1$, which are obtained by choosing a value for $x$ in $x^n$ and $x^n-1$, respectively. More generally, the product of two linear divisibility sequences, of integers or polynomials, is again a linear divisibility sequence, so it is natural to consider products $a_n=A\gamma^n\prod_i(\gamma_i^n-1)$ that generalize Lucas sequences. One such example was studied by Lehmer \cite{Leh}, who refined and extended Lucas's primality tests, and more by Ward \cite{Wd38}, back in 1930s. A general theory of such sequences was developed recently in \cite{RWG}, where one can also read more about their history and applications in algebra and number theory. Other old and new results on linear divisibility sequences and their applications can be found in the book \cite{EPS}. 

But $x^n$ and $x^n-1$ are just $x$ and $x-1$ composed with $x^n$. Classical authors thought that they generate ``almost all" linear divisibility sequences, with some special exceptions \cite{Wd38}, but they are themselves very special. One such exception, generated by composing $f(x)=(x+1)(x^m-1)$, $m$ odd, with $x^n$ is considered in \cite{Oos} (but was probably known already to Ward). What other  $f(x^n)$, with a polynomial $f$, are divisibility sequences? 
\begin{definition}
We call a polynomial $f(x)$ with complex coefficients a {\it divisibility polynomial} if $f(x^n)$ is a divisibility sequence, i.e., $f(x^m)|f(x^n)$ when $m|n$. It is enough that $f(x)|f(x^n)$ for all $n\in\mathbb{N}$, since then $f(x^m)|f(x^n)=f((x^m)^d)$ for $n=md$. 
\end{definition}
\noindent The divisibility polynomials are called Lucas polynomials in \cite{RWG}, but that name is already in use for other purposes. 
In 1988 B\'ezivin, Peth\"o and van der Poorten proved more generally \cite{BPP}, see also  \cite{Oos}, \cite{RWG}, that non-degenerate (in a precise sense) linear divisibility sequences of integers are always of the form 
\begin{equation}\label{BPPForm}
a_n=An^k\prod_if_i(\gamma_i^n),
\end{equation}
where $A$, $\gamma_i$ are complex numbers, $k$ is a nonnegative integer, and $f_i$ are divisibility polynomials. 

Not all sequences of this form are integer-valued, of course, and the integrality conditions are still not known in general. But we will fully describe the divisibility polynomials (Theorem \ref{DivPolyChar}), and explain where the ``exceptions" come from, and why they are more of a rule than exception. For example, if the positive integers $N_i$ are pairwise relatively prime then
\begin{equation}\label{DisDiv}
f(x)=\frac{(x^{N_1}-1)\cdots(x^{N_k}-1)}{(x-1)^{k-1}}
\end{equation}
is a divisibility polynomial. When $k=1$ we get the classical $f(x)=x-1$, and when $k=2,N_1=2$ and $N_2=m$ odd we get the ``exception" $f(x)=(x+1)(x^m-1)$. 

The divisibility polynomials have an interesting factorization theory (Sections \ref{S2'}-\ref{S3}), and can be nicely described as certain products of cyclotomic polynomials. The cyclotomic (literally, circle dividing) polynomials are classical, and also have ancient roots. Gauss introduced them in 1796 to solve an old Greek problem of inscribing regular polygons into the circle. Our description will be visual, we will introduce Hasse diagrams with additional labels that show how to build divisibility polynomials from cyclotomic polynomials (Section \ref{S2}). 

It turns out that Lucas's integer divisibility sequences generalize accordingly. Namely, if $f(x,y):=y^{\textrm{deg}\,(f)}f(x/y)$ is a homogenized divisibility polynomial, and $\alpha+\beta,\alpha\beta$ are integers, then $\frac{f(\alpha^n,\beta^n)}{f(\alpha,\beta)}$ is a divisibility sequence of integers (Theorem \ref{IntSeq}). When $f(x)=x-1$, or, more generally, $x^N-1$, we get Lucas's sequences. But for other $f(x)$, such as in \eqref{DisDiv}, the resulting sequences are {\it non-classical} -- they are not even products of Lucas's sequences. Nonetheless, they are closely related to them (Section \ref{S3'}). For the Fibonacci values of $\alpha,\beta$, the simplest one is $\frac12L_nF_{3n}$, where $L_n$ are the Lucas numbers. More generally, from the template \eqref{DisDiv} for $f(x)$ we get the sequence 
$$
\frac{F_{N_{1}n}/F_{N_{1}}\cdots F_{N_{k}n}/F_{N_{k}}}{(F_n)^{k-1}}\,.
$$

We will also characterize polynomials $f(x)$ with the property of {\it strong divisibility} (Theorem \ref{StrongDiv}), namely satisfying $\text{gcd}(f_m,f_n)=f_{\text{gcd}(m,n)}$, where $f_n(x)=f(x^n)$. A new connection between strong divisibility and cyclotomic polynomials was recently discovered in \cite{BFLS}, other known connections are discussed in \cite{Kimb}. This is another property that the Mersenne and Fibonacci numbers have in common, and it is distinctive of the classical sequences. Under some mild conditions on $\alpha,\beta$, our generalized sequences are strong divisibility sequences if and only if they are classical (Theorem \ref{StDiv}).

\section{Cyclotomic polynomials and Hasse diagrams}\label{S2}

The divisibility polynomials are clearly very special, but there are more of them than one might think. Let us call a polynomial {\it normal} if it has the leading coefficient $1$, and a non-zero constant term. Then any polynomial is of the form $Cx^sg(x)$ for some constant $C$, integer $s\geq0$, and normal $g(x)$. In this section we will develop diagrams that encode the structure of normal divisibility polynomials. 

We start with a couple of simple observations. Let $\zeta$ be a root of a divisibility polynomial $f$. Since $f(x)|f(x^n)$ and $f(\zeta)=0$ we have $f(\zeta^n)=0$ for all positive integers $n$. But a polynomial can not have infinitely many roots, so $\zeta^k=\zeta^l$ for $k\neq l$. This means that $\zeta$ is either $0$ or $\zeta^n=1$ for some $n\geq1$. Such $\zeta$ are called {\it $n$-th roots of unity}, and they are of the form $\zeta^{k}_n:=e^{\frac{2\pi i}{n}k}$. If $k\,|\,n$ then $\zeta^{k}_n=\zeta^1_{n/k}$ is also an $n/k$-th root of unity. Those roots that are not roots for smaller $n$ are called 
{\it primitive of order $n$}. By the elementary properties of cyclic groups, those are exactly the ones with $\text{gcd}(n,k)=1$. B\'ezivin, Peth\"o and van der Poorten \cite{BPP} found a necessary and sufficient condition on the set of roots that makes a polynomial a divisibility polynomial.
\begin{theorem}\label{DivRoot} Consider a polynomial $f$ with roots of unity as roots. Given a pair of them, let $h,h'$ denote their orders as primitive roots of unity, and $m,m'$ their multiplicities in $f$. Then $f$ is a divisibility polynomial if and only if for any such pair with $h'|h$ we have $m'\geq m$. 
\end{theorem} 
\begin{proof}First suppose that the roots of $f(x)$ satisfy the conditions of the theorem. We want to show that $f(x)|f(x^n)$. If $\zeta$ is a root of $f(x)$ of multiplicity $m$ then $\zeta':=\zeta^n$ is also a root, of multiplicity $m'\geq m$. Since $(x-\zeta^n)^{m'}$ is a factor of $f(x)$ we have that $(x^n-\zeta^n)^{m'}$ is a factor of $f(x^n)$. But $(x-\zeta)|(x^n-\zeta^n)$, so $(x-\zeta)^m|(x^n-\zeta^n)^{m'}$, as $m'\geq m$. Since this is true for all roots $f(x)|f(x^n)$.

Conversely, let $f$ be a divisibility polynomial, and $\zeta,\zeta'$ be a pair roots from the statement of the theorem. Since $h|h'$, by the standard properties of cyclic groups, there is $d\in\N$ such that $\zeta':=\zeta^d$. We have $(x-\zeta)^m$ dividing $f(x)$ and hence $f(x^d)$. But $f(x^d)$ splits into factors of the form $x^d-\xi$, where $\xi$ is a root of $f(x)$, and $(x-\zeta)$ must divide at least one of them. This is only possible if $\xi=\zeta^d=\zeta'$, and for $(x-\zeta)^m$ to divide $(x^d-\zeta^d)^{m'}$ we must have $m'\geq m$.
\end{proof}
\noindent Note that any primitive root of unity of some order is a power of any other primitive root of the same order. So if $\zeta$ is a root of $f(x)$ then so is every other primitive root of the same order, since $f(\zeta)|f(\zeta^n)$, and, by Theorem \ref{DivRoot}, they all have the same multiplicity. The product of $(x-\zeta)$ factors over all primitive roots $\zeta$ of order $n$ is called the {\it $n$-th cyclotomic polynomial} $\Phi_n(x)=\prod_{\text{gcd}(n,k)=1}(x-\zeta_n^k)$. It follows that a normal divisibility polynomial $f(x)$ is a product of powers of cyclotomic polynomials. 

Gauss derived from the definition of cyclotomic polynomials that
\begin{equation}\label{AllDiv}
x^n-1=\prod_{d|n}\Phi_d(x)\,.
\end{equation}
This is because any $n$-the root of unity must be primitive of some order that divides $n$. Gauss's formula allows us to find $\Phi_n$ once $\Phi_d$ for $d<n$ are already known, without any recourse to complex numbers, by splitting off $\Phi_n$ from the product on the right. It also shows, by the Gauss's lemma, that their coefficients are integers, and the leading coefficient is always $1$. It is easy to find them recursively from \eqref{AllDiv}: $\Phi_1(x)=x-1$, $\Phi_2(x)=x+1$, $\Phi_3(x)=x^2+x+1$, $\Phi_4(x)=x^2+1$, $\Phi_6(x)=x^2-x+1$, etc.

What kinds of cyclotomic products are the divisibility polynomials? Suppose $\Phi_{h}(x)$ appears as a factor in a divisibility polynomial. Theorem \ref{DivRoot} tells us that if $d|h$ then $\Phi_{d}(x)$ must also be a factor. In particular, $\Phi_{1}(x)=x-1$ is always a factor. We are led to the following property, known in algebra. 
\begin{definition}
A subset $\Lambda\subset\mathbb{N}$ is called {\it saturated} if $d|h$ and $h\in\Lambda$ imply that $d\in\Lambda$, i.e., if along with any of its elements $\Lambda$ also contains all of its divisors. 
\end{definition}
Divisibility relations among positive integers in a subset can be conveniently pictured by {\it Hasse diagrams}. Given a subset $\Lambda$, at the lowest level of the diagram we place $1$ (if it is in $\Lambda$), at the next level all prime numbers in $\Lambda$ (if any), next up are their pairwise products, and so on, see Figure\,\ref{MultiHasse}a. The edges connect the numbers to the numbers one level up which they divide, so upward paths in the diagram reflect the ``increase" in divisibility. It is convenient to assign the empty diagram to the empty set $\emptyset$.
\begin{figure}[!ht]
\begin{centering}
(a)\ \ \ \includegraphics[scale=0.9]{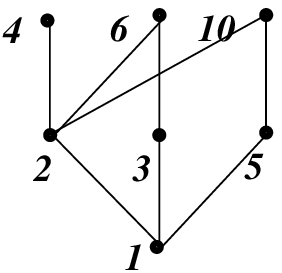} \hspace{0.7in} (b)\ \ \ \includegraphics[scale=0.9]{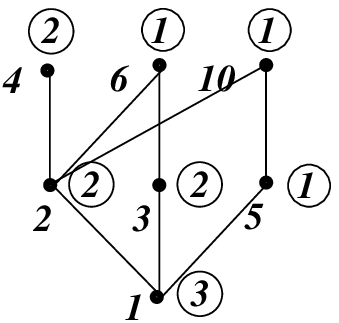}\par
\par\end{centering}
\caption{\label{MultiHasse}(a) Hasse diagram of $\{1,2,3,4,5,6,10\}$; (b) Hasse diagram of $\Phi_1^3\Phi_2^2\Phi_3^2\Phi_4^2\Phi_5\Phi_6\Phi_{10}$. The circled numbers are the multiplicities.}
\end{figure}

The Hasse diagrams do not reflect the multiplicities of factors $\Phi_h$ though. Those can be pictured as additional non-negative integer labels. Theorem \ref{DivRoot} imposes a restriction on these labels. For a product of cyclotomic polynomials $\prod_{h}\big(\Phi_h(x)\big)^{m_h}$ to be a divisibility polynomial, ``less divisible" primitive orders $h$ must have bigger multiplicities $m_h$. Since divisibility ``increases" along the upward paths in a Hasse diagram the multiplicity labels must not increase when moving up along the edges, Figure\,\ref{MultiHasse}b. To make this more precise, let us restate it in terms of maps. 
\begin{definition}
We call maps $\lambda:\N\to\N\cup\{0\}$, that are non-zero at only finitely many integers, {\it multiplicity maps}. And we call them {\it order-reversing} if 
$h|h'$ implies $\lambda(h)\geq\lambda(h')$. 
\end{definition}
\noindent Given a product of powers of cyclotomic polynomials as above, $\lambda:h\mapsto m_h$ is such a map. And conversely, any multiplicity map $\lambda$ defines a polynomial 
$$
\Phi_\lambda(x):=\prod_{h\in\N}\big(\Phi_h(x)\big)^{\lambda(h)}.
$$ 
This makes sense since all but finitely many factors have the exponent $0$. If we assign the polynomial $1$ to the zero map (and the empty diagram) then Theorem \ref{DivRoot} can be restated as follows.
\begin{corollary}\label{1-1Cor} There is a $1$-$1$ correspondence between order-reversing multiplicity maps, multiplicity labeled Hasse diagrams of finite saturated sets, and the normal divisibility polynomials. 
\end{corollary}
There is also a nice algebraic way to describe finite saturated sets that will come handy later. If $\Lambda$ is finite there must be $n_1,\dots,n_k\in\Lambda$ that do not divide any of its other elements. They are maximal in terms of divisibility, and they are, obviously, uniquely determined by $\Lambda$. But for finite saturated sets the converse is also true, because upward paths in a finite Hasse diagram must terminate at a maximal element. 
\begin{corollary}\label{<n>} Finite saturated subsets of $\N$ are of the form:
\begin{equation}
\langle n_1,\dots,n_k\rangle:=\{d\in\N\,\big|\,d\,|\,n_i\text{ for some }i\}\,.
\end{equation}
\end{corollary}
\noindent In particular, $\langle n\rangle$ is the set of all (positive) divisors of $n$. Figures \ref{MultiHasse} and \ref{MultDecomp} contain saturated sets $\langle 4,6,10\rangle$, $\langle pq\rangle$ and $\langle p,q,r\rangle$, for distinct primes $p,q,r$.

Finally, some notational conventions. Recall that the indicator function of $\Lambda$ is:
$$
1_\Lambda(h):=\begin{cases}1,& h\in\Lambda\\0,& h\not\in\Lambda\,.\end{cases}
$$
It is a map, and we abbreviate $\Phi_{1_\Lambda}$ as $\Phi_\Lambda$. This is simply the product $\prod_{d\in\Lambda}\Phi_d(x)$, and it means that $\Phi_{\langle n_1,\dots,n_k\rangle}(x)$ is the product of $\Phi_d(x)$, where $d$ runs over the set $\langle n_1,\dots,n_k\rangle$ of divisors of at least one of the positive integers $n_1,\dots,n_k$. They are natural generalizations of $\Phi_{\langle n\rangle}(x)=x^n-1$, and will play a major part in what follows.

\section{Factorization of divisibility polynomials}\label{S2'}

With the Hasse diagrams at hand, it is easier to take stock of the divisibility polynomials. We now realize that there are  many more of them than contemplated by the classical authors \cite{Wd38}. A natural idea is to break them up into simpler pieces in some way. Note that a product of divisibility polynomials is again a divisibility polynomial, in other words, divisibility polynomials form a set closed under multiplication. In such sets a natural way to decompose is to factor. Integers factor into primes, (general) polynomials factor into linear factors (over complex numbers), etc. And those do not factor any further.
\begin{definition}
A non-constant divisibility polynomial is called decomposable if it is the product of two non-constant
divisibility polynomials. Otherwise, it is called indecomposable.
\end{definition}
We do not yet know which divisibility polynomials are indecomposable, but here again we can make use of the diagrams. The ones associated with finite saturated sets look simpler than those with additional multiplicity labels on them. The idea of decomposing into them is pictured in Figure \ref{MultDecomp}a. We represent multiplicities by nodes stacked over the diagram, and then slice the layers of nodes horizontally. Each slice projects to a finite saturated set that defines a factor. This idea turns out to work. Recall that $\Phi_\lambda$ denotes the divisibility polynomial corresponding to the multiplicity map $\lambda$.
\begin{figure}[!ht]
\begin{centering}
(a)\ \ \ \includegraphics[width=0.21\textwidth]{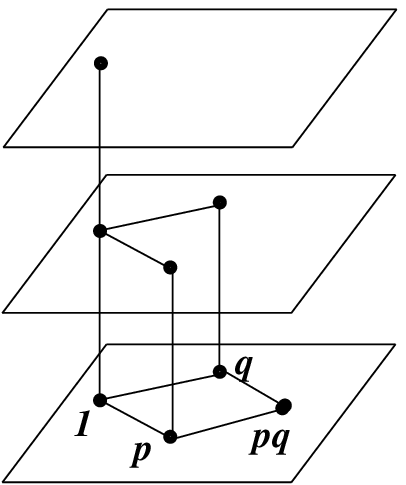} \hspace{0.7in} (b)\ \ \ \includegraphics[width=0.22\textwidth]{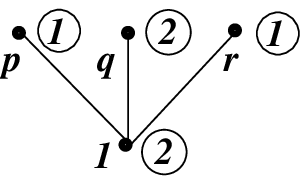}\par
\end{centering}
\caption{\label{MultDecomp} (a) Slicing decomposition of $\Phi_1^3\Phi_p^2\Phi_q^2\Phi_{pq}=\Phi_{\langle 1\rangle}\Phi_{\langle p,q\rangle}\Phi_{\langle pq\rangle}$; (b) Divisibility polynomial $\Phi_1^2\Phi_p\Phi_q^2\Phi_{r}$ with a non-unique decomposition ($p,q,r$ are distinct primes).}
\end{figure}
\begin{theorem}[{\bf Slicing factorization}]\label{Slice} A normal divisibility polynomial is indecomposable if and only if each of its cyclotomic factors has multiplicity $1$. It suffices that $\Phi_1(x)$ has multiplicity $1$. In other words, indecomposable normal divisibility polynomials are $\Phi_{\langle n_1,\dots,n_k\rangle}$ with $n_i\nmid n_j$ for $i\neq j$. Any normal divisibility polynomial $\Phi_\lambda$ factorizes into indecomposables: 
\begin{equation}\label{Lambda_j}
\displaystyle{\Phi_\lambda=\prod_{j=1}^{\lambda(1)}\Phi_{\Lambda_j}}\textrm{, where }\Lambda_j:=\{m\in\N\,|\,\lambda(m)\geq j\}\,.
\end{equation}
\end{theorem} 
\begin{proof} Saying that all cyclotomic factors of $\Phi_\lambda$ have multiplicity $1$ is equivalent to saying that $\lambda=1_\Lambda$ for a finite saturated set $\Lambda$, or that $\lambda$ is an order-reversing multiplicity map with $\lambda(1)=1$. 
Indeed, since $1|m$ and $\lambda$ is order-reversing $\lambda(m)\leq\lambda(1)$ for any $m\in\N$ . If $\lambda(1)=1$ then $\lambda(m)$ is $0$ or $1$. If $\lambda(1)=0$ then $\lambda$ is the zero map and $\Phi_\lambda=1$, otherwise it is $1_\Lambda$ for some finite saturated set $\Lambda$. Moreover, if $\Phi_\lambda=\Phi_\mu\Phi_\nu$ were a non-trivial factorization then $\lambda(1)=\mu(1)+\nu(1)\geq2$, so none exists if $\lambda(1)=1$.

Conversely, if $\lambda(1)\geq2$ we can factorize $\Phi_\lambda$ explicitly, see Figure \ref{MultDecomp}a. Let $\Lambda_j$ be the projection of the $j$-th layer of nodes above the diagram. Clearly, $\lambda=\sum_{j=1}^{\lambda(1)}1_{\Lambda_j}$, hence $\Phi_\lambda=\prod_{j=1}^{\lambda(1)}\Phi_{\Lambda_j}$. Moreover, if $m\in\Lambda_j$ and $d|m$ then $\lambda(d)\geq\lambda(m)\geq j$, so $d\in\Lambda_j$, and $\Lambda_j$ is a saturated set by definition.
\end{proof}

Although the slicing factorization has much to recommend itself, it is not unique. Consider the multiplicity map depicted in Figure\,\ref{MultDecomp}b. Its slicing factorization is $\Phi_{\langle q\rangle}\Phi_{\langle p,q,r\rangle}$, but $\Phi_{\langle p,q\rangle}\Phi_{\langle q,r\rangle}$ also factorizes it into indecomposables. The situation is not unlike the simple example of multiplication of $3n+1$ numbers $1,4,7,10,13,16,19,22$..., whose set is also closed under multiplication. We have $4\cdot55=10\cdot22$, and $4,22,10,55$ do not decompose into other $3n+1$ numbers. However, it is a simple exercise to show that if a $3n+1$ number has a $3n+1$ factor, then the quotient is also a $3n+1$ number. So such a number decomposes into a product of $3n+1$ numbers. Even this weaker property fails for normal divisibility polynomials. Every non-trivial one has $\Phi_{\langle 1\rangle}(x)=x-1$ as a factor, but plenty of them are indecomposable. For example, $\Phi_{\langle 2\rangle}(x)=x^2-1=(x-1)(x+1)$ is divisible by $\Phi_{\langle 1\rangle}(x)$, but does not factor into divisibility polynomials. The set of divisibility polynomials shows just how far multiplicatively closed sets can be from positive integers, where the prime factorization is unique.

\section{Compression of indecomposable polynomials}\label{S3}

Factorization of divisibility polynomials is a bit of a disappointment, even aside from non-uniqueness. Although $f_m(x)=x^m-1$ are all indecomposable, in a sense, they bring little new compared to $x-1$. Their divisibility sequences $f_m(x^n)=x^{mn}-1$ are rather boring subsequences of $x^{n}-1$, and if we pick a numerical value for $x$ the resulting numerical sequences are even more alike -- the picked value is simply replaced by its $m$-th power. We need to weed out this redundancy. 
\begin{definition}\label{PhixpDef}
A divisibility polynomial $f$ is called compressible into a divisibility polynomial $g$ if $f(x)=g(x^m)$, and it is called incompressible if no such $g$ exists. 
\end{definition}

To find out which $\Phi_M(x)$ can be compressed, i.e., written as $\Phi_\Lambda(x^n)$, we need to find out how $\Phi_\Lambda(x^n)$ {\it de}compress. We can start by decompressing their cyclotomic factors. One special case is straightforward by inspection of the roots on both sides: 
\begin{equation}\label{Phixp}
\Phi_m(x^p)=\begin{cases}\Phi_{mp}(x),\text{ if } p\,|\,m\\
\Phi_{m}(x)\Phi_{mp}(x),\text{ if } p\!\not|\,m\,,\end{cases}
\end{equation}
where $m$ is any positive integer, and $p$ is prime. For divisibility polynomials we can get a nicer formula, that does not split into cases. Recall that $\langle n\rangle:=\{d\in\N\,\big|\,d\,|\,n\}$, and $ST$ denotes the set of products of elements from sets $S$ and $T$. So $\langle n\rangle\Lambda$ consists of products of divisors of $n$ and elements of $\Lambda$.
\begin{theorem}[{\bf Decompression formula}]\label{TDecomp} Let $n\in\N$ and $\Lambda\subset\N$ be a finite saturated set. Then 
\begin{equation}\label{Decomp}
\Phi_\Lambda(x^n)=\Phi_{{\langle n\rangle\Lambda}}(x)=\prod_{d\in{\langle n\rangle\Lambda}}\Phi_d(x).
\end{equation}
\end{theorem} 
\begin{proof} By inspection, $\Phi_\Lambda(x^n)$ is a normal divisibility polynomial. Moreover, $x-1$ has multiplicity $1$ as a factor of $\Phi_\Lambda(x)$. No term of the form $x^n-\zeta$ is divisible by $x-1$, except when $\zeta=1$. Moreover, $x-1$ has multiplicity $1$ in $x^n-1$, and, therefore, $x-1$ has multiplicity $1$ in $\Phi_\Lambda(x^n)$ as well. So, by Theorem \ref{Slice}, $\Phi_\Lambda(x^n)$ is indecomposable, and $\Phi_\Lambda(x^n)=\Phi_{\Lambda^{(n)}}(x)$ for some finite saturated set, which, for the time being, we will denote $\Lambda^{(n)}$. It remains to show that $\Lambda^{(n)}=\langle n\rangle\Lambda$. 

First suppose that $n=p$ is prime. If $d\in\Lambda^{(p)}$ then $\Phi_d(x)$ is a factor of $\Phi_\Lambda(x^p)$. We see from \eqref{Phixp} that those are of the form $\Phi_{m}(x)$ and $\Phi_{mp}(x)$ with $m\in\Lambda$, so $d\in\langle p\rangle\Lambda$. Conversely, if $m\in\Lambda$ we need to show that $m,mp\in\Lambda^{(p)}$. That $\Phi_{mp}(x)$ is a factor of $\Phi_m(x^p)$ is immediate from \eqref{Phixp}, as it is for $\Phi_{m}(x)$ when $p\nmid m$. But if $p\,|\,m$ then $\frac{m}{p}\in\Lambda$ by the saturation property, and $\Phi_{m}(x)$ is always a factor of $\Phi_{\frac{m}{p}}(x^p)$. Hence, $m\in\Lambda^{(p)}$ also in this case. Thus, 
$\Lambda^{(p)}=\langle p\rangle\Lambda$ in both cases.

To finish the proof, we need a simple observation about $\langle mn\rangle$. If $d|mn$ then, factorizing all three into primes, we can find $d_1|m$ and $d_2|n$ such that $d=d_1d_2$. Conversely, if $d_1|m$ and $d_2|n$ then $d_1d_2|mn$. Summarizing, $\langle mn\rangle=\langle m\rangle\langle n\rangle$. The general case now follows by induction on the number of prime factors in $n=p_1\cdots p_N$ (some $p_i$ may repeat) since $\Phi_\Lambda(x^{(p_1\cdots p_{k})p_{k+1}})=\Phi_{\langle p_{k+1}\rangle\Lambda}(x^{p_1\cdots p_{k}})$ and $\langle p_{k+1}\rangle\langle p_1\cdots p_{k}\rangle=\langle p_1\cdots p_{k+1}\rangle$.
\end{proof}
\noindent A warning: even if $d$ can be represented as $d_1d_2$, with $d_1|n$ and $d_2\in\Lambda$, in several different ways $\Phi_d(x)$ still enters the product in \eqref{Decomp} only once.
\begin{figure}[!ht]
\begin{centering}
\includegraphics[scale=0.9]{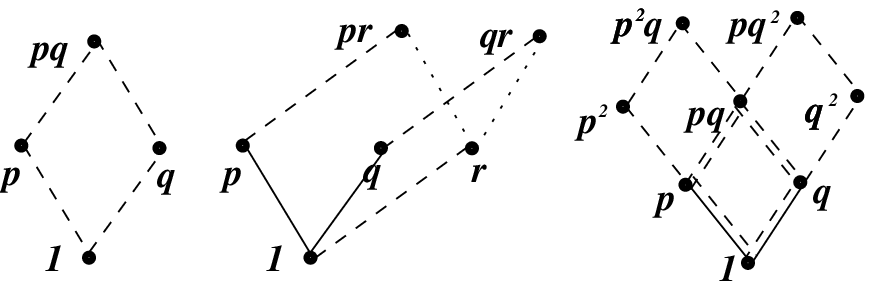}
\par\end{centering}
\caption{\label{DecomDiag}Decompression diagrams for $\Phi_{\langle1\rangle}(x^{pq})$, $\Phi_{\langle p,q\rangle}(x^r)$ and 
$\Phi_{\langle p,q\rangle}(x^{pq})$, $p,q,r$ are distinct primes. Adjoined (possibly twice) edges are dashed, additional edges are dotted.}
\end{figure}

When $\Lambda=\langle1\rangle=\{1\}$ the decompression formula reduces to the Gauss's $x^n-1=\prod_{d\,|\,n}\Phi_d(x)$. Geometrically, the Hasse diagram of $\langle n\rangle\Lambda$ is obtained by adjoining a copy of $\langle n\rangle$ to every node of $\Lambda$ (or equivalently a copy of $\Lambda$ to every node of $\langle n\rangle$), and possibly adding extra edges according to the diagram drawing rules, see Figure\,\ref{DecomDiag}. This suggests that to compress a diagram we need to look for a subdiagram in the shape of $\Lambda$, to which identical diagrams in the shape of $\langle n\rangle$ are attached. This is a non-trivial search for large diagrams. 
\begin{figure}[!ht]
\begin{centering}
\includegraphics[scale=0.8]{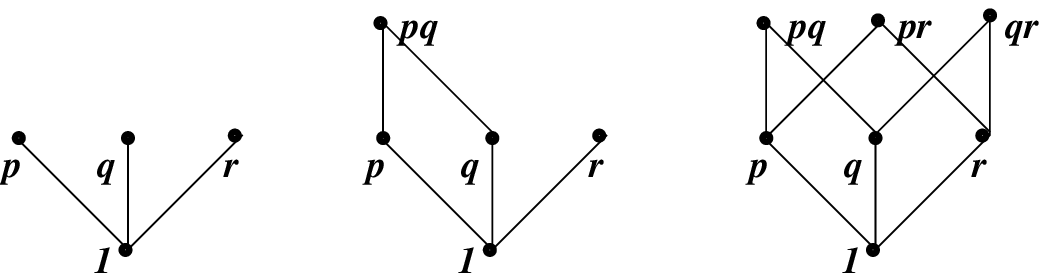}
\par\end{centering}
\caption{\label{Incomp}Incompressible diagrams, $p,q,r$ are distinct primes.}
\end{figure}

Fortunately, there is a much simpler way to detect compressibility. The reader may try to guess it by reflecting on the difference between the diagrams in Figures \ref{DecomDiag} and \ref{Incomp}. This is one of those cases where tricky geometry is streamlined by simple algebra. Since one can easily show that $\langle d\rangle\langle n_1,\dots,n_k\rangle=\langle dn_1,\dots,dn_k\rangle$ (by the same argument as for $\langle m\rangle\langle n\rangle=\langle mn\rangle$) we can rewrite \eqref{Decomp} as 
$$
\Phi_{\langle m_1,\dots,m_k\rangle}(x^d)=\Phi_{\langle dm_1,\dots,dm_k\rangle}(x).
$$ 
So $\Phi_{\langle n_1,\dots,n_k\rangle}$ is compressible if and only if $n_i$ have a non-trivial common divisor. 
\begin{corollary} $\Phi_{\langle n_1,\dots,n_k\rangle}(x)$ is incompressible if and only if $\textup{gcd}(n_1,\dots,n_k)=1$.  Any indecomposable normal divisibility polynomial is of the form $\Phi_{\langle n_1,\dots,n_k\rangle}(x^m)$, with $m\geq1$ and an incompressible $\Phi_{\langle n_1,\dots,n_k\rangle}(x)$. 
\end{corollary} 
Now it is easy to give explicit ``exceptional" examples of divisibility polynomials that are distinctive. If $n_1, n_2$ are relatively prime then $x^{n_1}-1=\prod_{d\,|\,n_1}\Phi_d(x)$ and $x^{n_2}-1=\prod_{d\,|\,n_2}\Phi_d(x)$ have only one common cyclotomic factor, $\Phi_1(x)=x-1$, so 
$$
\Phi_{\langle n_1,n_2\rangle}(x)=\frac{(x^{n_1}-1)(x^{n_2}-1)}{x-1}\,.
$$
Similarly, if $\textrm{gcd}(n_1,n_2,n_3)=1$ then
\begin{equation}\label{Phi3Prime}
\Phi_{\langle n_1,n_2,n_3\rangle}(x)=\frac{(x^{n_1}-1)(x^{n_2}-1)(x^{n_3}-1)(x-1)}{(x^{\textrm{gcd}(n_1,n_2)}-1)(x^{\textrm{gcd}(n_2,n_3)}-1)
(x^{\textrm{gcd}(n_3,n_1)}-1)}\,.
\end{equation}
In general, one can use the inclusion-exclusion formula from combinatorics to express $\Phi_{\langle n_1,\dots,n_k\rangle}(x)$ for any $k$. When $\textrm{gcd}(n_i,n_j)=1$ for $i\neq j$ it simplifies to
\begin{equation}\label{PhikPrime}
\Phi_{\langle n_1,\dots,n_k\rangle}(x)=\frac{(x^{n_1}-1)\cdots(x^{n_k}-1)}{(x-1)^{k-1}}\,.
\end{equation}
Let us summarize what we discovered about the structure of the divisibility polynomials.
\begin{theorem}\label{DivPolyChar} Any divisibility polynomial $f(x)$ is of the form $f(x)=Cx^sg(x)$, where $C$ is a numerical constant, $s$ is a positive integer, and $g(x)$ is a product of factors of the form $\Phi_{\langle n_1,\dots,n_k\rangle}(x^d)$, with integers $d\geq0,n_i\geq1$, $n_i\nmid n_j$ for $i\neq j$, and $\textup{gcd}(n_1,\dots,n_k)=1$.
\end{theorem} 

\section{Back to Mersenne and Fibonacci}\label{S3'}

We now have a characterization of the divisibility polynomials, and it is time to apply the fruits of our labor. To fully characterize linear divisibility sequences of integers, it ``only" remains to find conditions on $A$ and $\gamma_i$ in \eqref{BPPForm}, that would produce integer sequences. Unfortunately, this is a big only -- necessary and sufficient integrality conditions are not known even when all divisibility polynomials are $x$ or $x-1$ \cite{RWG}. We will find sufficient integrality conditions when there is a single divisibility polynomial factor, and construct many non-classical relatives of the Fibonacci and Mersenne numbers. 

It will be convenient to represent our sequences in a different form than \eqref{BPPForm}, using a  modification of the cyclotomic polynomials into polynomials in two variables. The modification is common in algebra, and is called {\it homogenization}: given a polynomial $f(x)$ we set $f(x,y):=y^{\textrm{deg}\,(f)}f(x/y)$. It leaves $x$ as is, but turns $x-1$ into $x-y$, and $x^n-1$ into $x^n-y^n$. If $f(x)$ is a product of polynomials then its homogenization is the product of their homogenizations. In particular, $\Phi_\Lambda(x,y):=\prod_{d\in\Lambda}\Phi_d(x,y)$, where $\Phi_d(x,y)$ are the homogenized cyclotomic polynomials: $\Phi_1(x,y)=x-y$, $\Phi_2(x,y)=x+y$, $\Phi_3(x,y)=x^2+xy+y^2$, $\Phi_4(x,y)=x^2+y^2$, $\Phi_6(x,y)=x^2-xy+y^2$, etc.

A key observation is that $\Phi_d(x,y)$ are symmetric in $x$ and $y$ for $d>1$. This is because $\Phi_d(x)$ for $d>1$ are {\it palindromic}, their coefficients read the same forward and backward. To prove this, note that $\frac{x^n-1}{x-1}$ is palindromic (all coefficients are $1$), and if $f(x)g(x)=h(x)$, where $f$ and $h$ are palindromic, then so is $g$. Dividing the Gauss's formula \eqref{AllDiv} by $x-1$ we have
\begin{equation}\label{AllDivx1}
\frac{x^n-1}{x-1}=\Phi_n(x)\prod_{d|n,1<d<n}\Phi_d(x)\,,
\end{equation}
Assuming $\Phi_d(x)$ are palindromic for $d<n$, $\Phi_n(x)$ is a quotient of palindromic polynomials, hence itself palindromic. By induction, $\Phi_n(x)$ are palindromic for all $n>1$. 

Now it is easy to figure out the secret behind $F_n=\frac{\alpha^n-\beta^n}{\alpha-\beta}$ being integers, despite the irrational values $\alpha=\frac{1+\sqrt{5}}2$ and $\beta=\frac{1-\sqrt{5}}2$.  By the fundamental theorem on symmetric polynomials, any such polynomial with integer coefficients is also a polynomial with integer coefficients in $x+y$ and $xy$, and $\alpha+\beta=1$, $\alpha\beta=-1$ are integers. We can generalize this integrality argument as follows.
\begin{theorem}\label{IntSeq} Let $\alpha,\beta$ be complex numbers such that $\alpha+\beta$ and $\alpha\beta$ are integers, and $f(x)$ be a normal divisibility polynomial. Then $A_n:=\frac{f(\alpha^n,\beta^n)}{f(\alpha,\beta)}$ is a linear divisibility sequence of integers.
\end{theorem} 
\begin{proof} Since normal divisibility polynomials are products of indecomposables, and products of integer sequences are integer sequences, it suffives to consider an indecomposable $f=\Phi_\Lambda$. By the decompression formula, 
$\frac{\Phi_\Lambda(x^n)}{\Phi_\Lambda(x)}=\Phi_{{\langle n\rangle\Lambda}}(x)/\Phi_\Lambda(x)$ is a product of cyclotomic polynomials. In particular, it is a polynomial with integer coefficients. Therefore, so is
$$
\frac{\Phi_\Lambda(x^n,y^n)}{\Phi_\Lambda(x,y)}=\frac{\prod_{d\in\Lambda}\Phi_d(x^n,y^n)}{\prod_{d\in\Lambda}\Phi_d(x,y)}=
\frac{x^n-y^n}{x-y}\prod_{d\in\Lambda,\,d>1}\frac{\Phi_d(x^n,y^n)}{\Phi_d(x,y)}\,.
$$
Moreover, the first factor is symmetric in $x$ and $y$, because it homogenizes a palindromic $\frac{x^n-1}{x-1}$, and the second factor is also symmetric, because $\Phi_d(x,y)$ are for $d>1$. By the fundamental theorem on symmetric polynomials, the product is a symmetric polynomial with integer coefficients in $x+y$ and $xy$, and its values are integers for $x=\alpha,y=\beta$. By the same reasoning, $\frac{\Phi_\Lambda(\alpha^n,\beta^n)}{\Phi_\Lambda(\alpha^m,\beta^m)}$ is an integer when $m|n$, so $A_m|A_n$ in $\Z$.
\end{proof}
As a first example, consider the simplest non-classical divisibility polynomial 
$$
f=\Phi_{\langle 2,3\rangle}=\Phi_1\Phi_2\Phi_3=(x-y)(x+y)(x^2+xy+y^2)=(x+y)(x^3-y^3)\,.
$$
The corresponding divisibility sequence is 
$$
A_n=\frac{x^n+y^n}{x+y}\cdot\frac{x^{3n}-y^{3n}}{x^3-y^3}=B_n\cdot C_n\,.
$$ 
Both factors satisfy linear recurrences of the second order, and can be generated easily:\\
\hspace*{1in} $B_0=\frac2{x+y},\ B_1=1,\ B_{n+2}=(x+y)B_{n+1}-xyB_n;$\\
\hspace*{1in} $C_0=0,\ C_1=1,\ C_{n+2}=(x^3+y^3)C_{n+1}-x^3y^3C_n.$\\
For the Fibonacci values $\alpha,\beta$, with $\alpha+\beta=1$ and $\alpha\beta=-1$, $B_n$ are none other than the Lucas numbers $L_n$, and $C_n$ is a normalized subsequence of the Fibonacci numbers: $\frac{F_{3n}}{F_3}=\frac12F_{3n}$. So $A_n=\frac12L_nF_{3n}$, see Table \ref{LucFib}. This construction can be generalized by picking any odd number $N$ instead of $3$, the resulting sequence is $A_n=L_nF_{N\!n}/F_N$, recurrent of order four. Some of these sequences were entered into the On-Line Encyclopedia of Integer Sequences by Bala in 2014 \cite{Bala}. 
\begin{table}[!ht]
\centering
\begin{tabular}{|l|llllllllll|}
\hline
$n$   & 0 & 1 & 2 & 3 & 4 & 5 & 6 & 7 & 8 & 9\\ \hline
$L_n$ & 2 & 1 & 3 & 4 & 7 & 11 & 18 & 29 & 47 & 76 \\ \hline
$F_{3n}/F_3$ & 0 & 1 & 4 & 17 & 72 & 305 & 1,292 & 5,473 & 23,184 & 98,209\\ \hline
$A_n$ & 0 & 1 & 12 & 68 & 504 & 3355 & 23,256 & 158,717 & 1,089,648 & 7,463,884\\ \hline
\end{tabular}
\caption{\label{LucFib} Non-classical linear divisibility sequence constructed from the Fibonacci numbers.}
\end{table}
\vspace{-1em}

\noindent A recurrent sequence of order six comes from $\Phi_{\langle 3,4\rangle}(x)=(x+1)(x^2+1)(x^3-1)$. For the Fibonacci values, we find $A_n=L_n\frac{L_{2n}}{L_2}\frac{F_{3n}}{F_3}=\frac16L_nL_{2n}F_{3n}$. Note that the extra factor $\frac{L_{2n}}{L_2}=2/3,1,7/3,6,47/3,41\dots$ not only is not a divisibility sequence, but is not even integer-valued.

To generalize them further, let us define 
$$
S_n^\Lambda(\alpha,\beta):=\frac{\Phi_\Lambda(\alpha^n,\beta^n)}{\Phi_\Lambda(\alpha,\beta)},
$$
We will abbreviate $S_n^{\langle 1\rangle}$ as $S_n$, so $M_n=S_n(2,1)$, $F_n=S_n\big(\frac{1+\sqrt{5}}2,\frac{1-\sqrt{5}}2\big)$, and suppress $\alpha,\beta$ from the notation when they are unimportant or understood. Any second order recurrent sequence with $S_0=0$ and $S_1=1$ is of this form, with $\Lambda=\langle 1\rangle$ and suitable $\alpha,\beta$. When $\Lambda=\langle N\rangle$ we get normalized subsequences of $S_n$, namely $S_n^{\langle N\rangle}=S_{N\!n}/S_{N}$. More generally, we can use the expression \eqref{PhikPrime} for $\Phi_\Lambda$. In the homogenized form, when $\textrm{gcd}(N_i,N_j)=1$ for $i\neq j$, we have:
$$
\Phi_{\langle N_1,\dots,N_k\rangle}(x,y)=\frac{(x^{N_1}-y^{N_1})\cdots(x^{N_k}-y^{N_k})}{(x-y)^{k-1}}\,.
$$
Therefore, 
\begin{equation}\label{SnNk}
S_n^{\langle N_1,\dots,N_k\rangle}=\frac{S_{N_{\!1}n}/S_{N_{\!1}}\cdots S_{N_{\!k}n}/S_{N_{\!k}}}{(S_n)^{k-1}}\,.
\end{equation}
We recover the sequences $A_n=L_nF_{N\!n}/F_N$, when $S_n=F_n$, from the well-known identity $L_n=F_{2n}/F_n$, and $F_2=1$.We leave it to the reader to produce more general sequences based on the expressions like \eqref{Phi3Prime}. When $\Lambda\neq\langle N\rangle$ these sequences are non-classical, and not covered by the extension of Lucas's theory developed in \cite{RWG}. Yet, all sequences from \eqref{SnNk} are divisors of a product of classical sequences. This is not accidental -- it is true for all non-degenerate linear divisibility sequences by a main result of \cite{BPP}.

\section{Strong divisibility}\label{S4}

The Fibonacci and Mersenne numbers have even more in common than being divisibility sequences. They are both {\it strong divisibility} sequences: $\textrm{gcd}(a_n,a_m)=a_{\textrm{gcd}(n,m)}$. This follows by a clever application of the Euclidean algorithm, for example. We will show that this property distinguishes them from their non-classical generalizations.

Let us start with the simplest non-classical sequence $f_n=\Phi_{\langle2,3\rangle}(x^n)$. By the decompression formula,\\
\hspace*{1in} $f_1=\Phi_1\Phi_2\Phi_3$\\
\hspace*{1in} $f_2=\Phi_1\Phi_2\Phi_3\Phi_4\Phi_6$\\
\hspace*{1in} $f_3=\Phi_1\Phi_2\Phi_3\Phi_6\Phi_9$\,.\\
\vspace{-0.5em}

\noindent Since $\Phi_d$ with different $d$ have no common roots $\textrm{gcd}(f_2,f_3)=\Phi_1\Phi_2\Phi_3\Phi_6\neq\Phi_1\Phi_2\Phi_3=f_1=f_{\textrm{gcd}(2,3)}$, so this sequence is not a strong divisibility sequence. It is natural to ask when $\Phi_\Lambda(x^n)$ is a strong divisibility sequence.

Recall that $\langle n\rangle$ denotes the set of all positive divisors of $n$, and $ST$ the set of products of elements from sets $S$ and $T$. Generalizing the above example, we see that the common cyclotomic factors of $\Phi_\Lambda(x^m)=\Phi_{\langle m\rangle\Lambda}(x)$ and $\Phi_\Lambda(x^n)=\Phi_{\langle n\rangle\Lambda}(x)$ are those with the indices from $\langle m\rangle\Lambda\cap\langle n\rangle\Lambda$. Therefore, $\Phi_\Lambda(x^n)$ is a strong divisibility sequence if and only if 
\begin{equation}\label{StDivSet}
\langle m\rangle\Lambda\cap\langle n\rangle\Lambda=\langle\textrm{gcd}(m,n)\rangle\Lambda\,.
\end{equation}
But we can give a much more explicit characterization.
\begin{theorem}\label{StrongDiv} Let $\Lambda\subset\N$ be a finite saturated set. Then $\Phi_\Lambda(x^n)$ is a strong divisibility sequence of polynomials if and only if $\Lambda=\langle N\rangle$ for some $N\geq1$.
\end{theorem} 
\begin{proof}First, suppose $\Lambda=\langle N\rangle$. Then we have 
$$
\langle m\rangle\langle N\rangle\cap\langle n\rangle\langle N\rangle=\langle mN\rangle\cap\langle nN\rangle
=\langle\textrm{gcd}(mN,nN)\rangle,
$$
because every common divisor of two numbers is a divisor of their $\textrm{gcd}$, and vice versa. Multiplying every line of the Euclidean algorithm on $m,n$ by $N$, we may conclude that $\textrm{gcd}(mN,nN)=\textrm{gcd}(m,n)N$. Therefore, 
$\langle\textrm{gcd}(mN,nN)\rangle=\langle\textrm{gcd}(m,n)\rangle\langle N\rangle$, and \eqref{StDivSet} holds, so $\Phi_\Lambda(x^n)$ is a strong divisibility sequence.

To prove the converse, we will first show that if \eqref{StDivSet} holds, and $m,n\in\Lambda$, then $\textrm{lcm}(m,n)\in\Lambda$. Clearly, $mn\in\langle m\rangle\Lambda\cap\langle n\rangle\Lambda=\langle\textrm{gcd}(m,n)\rangle\Lambda$. Therefore, there is $d\,|\,\textrm{gcd}(m,n)$ and $h\in\Lambda$ such that $mn=dh$. But then,
$$
h=\frac{mn}{\textrm{gcd}(m,n)}\cdot\frac{\textrm{gcd}(m,n)}{d}=\textrm{lcm}(m,n)\cdot\frac{\textrm{gcd}(m,n)}{d},
$$
so $\textrm{lcm}(m,n)\,|\,h$. Since $h\in\Lambda$, and $\Lambda$ is saturated, this implies $\textrm{lcm}(m,n)\in\Lambda$. As this is true for any $m,n\in\Lambda$, we have $\Lambda=\langle N\rangle$, where $N=\textrm{lcm}\{h\,|\,h\in\Lambda\}$.
\end{proof}
\begin{corollary} Every indecomposable divisibility polynomial that generates a strong divisibility sequence is compressible into $x-1$, i.e. it is of the form $x^{N}-1$.
\end{corollary} 
\noindent In other words, strong divisibility sequences of polynomials are the classical ones. Does it extend to the corresponding integer sequences? The question is subtle because we may get $\pm1$ when substituting values into $\Phi_d(x,y)$, even though the polynomial itself is non-trivial. Barring such degeneracy, we can characterize which linear divisibility sequences defined in Section \ref{S3'} are strong divisibility sequences. 
\begin{theorem}\label{StDiv} Let $\alpha,\beta$ be complex numbers such that $\alpha+\beta$ and $\alpha\beta$ are integers, $\alpha/\beta$ is not a root of unity, and $|\Phi_d(\alpha,\beta)|\neq1$ for all large enough $d$. Then $S_n^\Lambda(\alpha,\beta)$ is a strong divisibility sequence of integers if and only if $\Lambda=\langle N\rangle$ for some $N\geq1$.
\end{theorem} 
\begin{proof}If $\Lambda=\langle N\rangle$ the strong divisibility follows directly from Theorem \ref{StrongDiv}, even for the polynomials $\Phi_\Lambda(x^n,y^n)$. If $\Lambda\neq\langle N\rangle$, by the same argument as in the proof of the theorem,  there are $m,n\in\Lambda$ such that for any positive integer $k$ we have $kmn\in\langle km\rangle\Lambda\cap\langle kn\rangle\Lambda$, but $kmn\not\in\langle\textrm{gcd}(km,kn)\rangle\Lambda$. Therefore, $\textrm{gcd}(S_{km}^\Lambda,S_{kn}^\Lambda)$ multiplies $S_{\textrm{gcd}(km,kn)}^\Lambda$ by at least $\Phi_{kmn}(\alpha,\beta)$. Note that $S_n^\Lambda(\alpha,\beta)$ can only be $0$ for $n>0$ if $\alpha/\beta$ is a root of unity, and, by assumption, $\Phi_{kmn}(\alpha,\beta)\neq\pm1$ for large enough $k$. Thus, $S_n^\Lambda$ is not a strong divisibility sequence.
\end{proof}
\noindent For the sequences with real $\alpha,\beta$, like the Mersenne and Fibonacci numbers, we can rule out the numerical degeneracy easily.
\begin{corollary}\label{PhiReal} Let $\alpha,\beta$ be real numbers such that $\alpha+\beta$ and $\alpha\beta$ are integers, and $\alpha\beta\neq0$, $\alpha\neq\pm\beta$. Then $|\Phi_d(\alpha,\beta)|>1$ for $d>2$, and the only strong divisibility sequences among $S_n^\Lambda$ are the classical ones 
$S_{N\!n}/S_{N}$. 
\end{corollary} 
\begin{proof} From definition, $\Phi_d(x,y)$ is the product of $x-\zeta y$, where $\zeta$ runs over all primitive roots of unity of order $d$. Consider two circles centered at the origin with the radii $|\alpha|$ and $|\beta|$. Then $\zeta\beta$ is on the second circle, since $|\zeta|=1$, and its distance $|\alpha-\zeta\beta|$ to $\alpha$ is no less than the distance between the circles, $|\alpha+\beta|$ or $|\alpha-\beta|$, depending on the signs of $\alpha,\beta$. Moreover, since $\alpha,\beta$  are real the distance can equal $|\alpha\pm\beta|$ only if $\zeta=\pm1$, , i.e.  the order of $\zeta$ is $1$ or $2$. But $\alpha+\beta$ is an integer, so $|\alpha+\beta|\geq1$, unless $\alpha=-\beta$. Similarly, $(\alpha-\beta)^2=(\alpha+\beta)^2-4\alpha\beta$ is an integer, so $|\alpha-\beta|\geq1$, unless $\alpha=\beta$. Thus, $|\alpha-\zeta\beta|>1$ for $\zeta\neq\pm1$, and $|\Phi_d(\alpha,\beta)|>1$ for $d>2$. The second conclusion follows from Theorem \ref{StDiv}.
\end{proof}
\noindent This implies that of the divisibility sequences we constructed in Section \ref{S3'} the only strong divisibility sequences are the classical ones, $S_n$ and $S_{N\!n}/S_{N}$. One can show that the conditions of Theorem \ref{StDiv} also hold for complex conjugate $\alpha,\beta$, when $\alpha\beta\neq0$ and $\alpha/\beta$ is not a root of unity. But the proof relies on a non-elementary estimate of Stewart \cite{Stew} for $|\Phi_d(\alpha,\beta)|$, and we omit it.  
\bigskip

\noindent {\small {\bf Acknowledgements:} } The author is grateful to Brian Pasko and anonymous referees for their comments on a previous version of this paper, and to Lawrence Somer for valuable corrections to the final version.

\end{document}